\pdfoutput=1
\documentclass[a4paper,12pt,reqno,twoside]{amsart}
\usepackage{amsmath}
\usepackage{amsfonts}
\usepackage{amssymb}
\usepackage{amsthm}
\usepackage{array}
\usepackage{float}
\usepackage{hyperref}% http://ctan.org/pkg/hyperref
\hypersetup{hidelinks}
\usepackage{tabularx}
\usepackage[numbers,sort&compress]{natbib}
\newcolumntype{L}{>{\raggedright\arraybackslash}X}
\newcolumntype{P}[1]{>{\raggedright\arraybackslash}p{#1}}

\numberwithin{equation}{section}
\newtheorem{definition}{Definition}[section]
\newtheorem{theorem}[definition]{Theorem}

\newtheorem{remark}[definition]{Remark}

\newcommand*{\R}{\mathbb{R}}%.............................R
\newcommand{\comment}[1]{}
\title[Unified Boussinesq--Euler boundary-jet system]%
      {A unified Boussinesq--Euler formulation and finite-time blow-up for a Hou--Luo type boundary-jet system} 
\author[Yaoming Shi]{Yaoming Shi}
\address{California, United States}
\email{ymshi@protonmail.com}
\date{May 5, 2026}
\AtBeginDocument{%
\begin{abstract}
We derive a unified vorticity--stream formulation $(Bm)$ for two parity-reduced inviscid systems in the meridian plane: the 2D inviscid Boussinesq equations $(m=1)$ and the 3D axisymmetric Euler equations with swirl $(m=2)$.  In the Boussinesq case we set $\Theta=\vartheta/r$ and write $\Theta=u^2$ only when a smooth square-root branch has been fixed; equivalently, one may keep the scalar variable $\Theta$ throughout.  In the squared radial variable $q=r^2$, the two cases are encoded by the same parameterized system with $m=1,2$.  At the boundary $q=1$, a Taylor expansion gives an exact boundary jet: the transport equations close on the boundary, while the elliptic relation also contains the next normal jet $\varphi_{qq}(x,1,t)$.  If the boundary jet is closed by the first-order Taylor truncation $\varphi_{qq}(x,1,t)=0$, it reduces to a closed unified $(1+1)$D system $(Q0)$ with the local boundary velocity law $u=-(m+2)^{-1}\omega$.  We prove finite-time blow-up for this closed Hou--Luo type model on a periodic interval by a Riccati argument in the spirit of Choi--Hou--Kiselev--Luo--\v{S}ver\'ak--Yao.  The theorem is therefore a blow-up result for the closed boundary-jet model, not for the unrestricted Boussinesq or Euler systems.
\end{abstract}
\maketitle
}
\begin{document}
%\tableofcontents
%..........................................
%
%
%\section{\textbf{Pre-introduction}}

\section[Introduction]{Introduction}
The formation of finite-time singularities for the two-dimensional inviscid Boussinesq equations and for the three-dimensional incompressible Euler equations with swirl remains a central open problem in mathematical fluid dynamics.  This paper isolates a common reduced structure behind these two settings.  Under natural parity assumptions in the meridian plane and after the change of variable
\[
        q=r^2,
\]
we derive a unified vorticity--stream system $(Bm)$.  The case $m=1$ corresponds to the 2D inviscid Boussinesq equations, written either in the scalar variable $\Theta=\vartheta/r$ or, when a smooth square-root branch is fixed, in the notation $\Theta=u^2$, while $m=2$ corresponds to the 3D axisymmetric Euler equations with swirl.  In this formulation the two equations are encoded by the same symbolic system with coefficients depending on $m$.

The second part of the paper studies what $(Bm)$ induces at the outer boundary $q=1$, the setting associated with Hou--Luo type boundary growth.  The Taylor expansion at $q=1$ gives an exact boundary jet.  Its transport equations close on the boundary, but the elliptic relation contains the next normal jet $\varphi_{qq}(x,1,t)$.  Thus the unrestricted boundary trace is not a closed $(1+1)$D system by itself.  If the Taylor jet is closed at first order in the normal variable by imposing
\[
        \varphi_{qq}(x,1,t)=0,
\]
the boundary jet reduces to a closed unified system $(Q0)$.  This system has the same transport equations as the Hou--Luo and CKY models, but its Biot--Savart law is the local algebraic relation
\[
        u=-\frac1{m+2}\omega .
\]
We then prove finite-time blow-up for smooth periodic data satisfying the standard symmetry and sign assumptions used in this class of one-dimensional models.  This blow-up theorem applies to the closed boundary-jet model $(Q0)$, not to the unrestricted boundary trace of $(Bm)$ or to the full Boussinesq/Euler systems.  The same proof also covers the one-parameter linear normal-jet closure discussed in Remark~\ref{rem:linear-normal-jet-closure}.

\paragraph{\textbf{Related work and context.}}
The mathematical literature around singularity formation for inviscid fluids is extensive, and we recall only the works most directly connected with the reductions and blow-up mechanisms used here.  Classical continuation and loss-of-regularity perspectives for the 3D Euler equations include Beale--Kato--Majda~\citep{BKM984} and Constantin~\citep{C1986,C2007}.  For the inviscid 2D Boussinesq system, local theory and conditional breakdown criteria go back to Cannon--DiBenedetto~\citep{CD1980}, Chae--Nam~\citep{ChaeNam1997}, Chae--Kim--Nam~\citep{CKN1999}, and Taniuchi~\citep{Tan2002}; see also Wu's lecture notes~\citep{Wu2012} and the small-scale formation work of Kiselev--Park--Yao~\citep{KPY2022}.  For the axisymmetric Euler geometry, the pressure--velocity and vorticity-building-block viewpoint is naturally compared with the standard Euler and axisymmetric formulations in Majda--Bertozzi~\citep{MB2002}, Hou--Li~\citep{HLi2006}, Chae--Lee~\citep{CL2002}, Drazin--Riley~\citep{DR2006}, and Chen--Fang--Zhang~\citep{CFZ2015}.

Model problems and exactly solvable mechanisms provide an important guide for the present boundary dynamics.  The Constantin--Lax--Majda model~\citep{CLM1985}, the De~Gregorio model~\citep{DeGregorio1990}, Schochet's viscous model analysis~\citep{Schochet1986}, the didactic 2D model of Chae--Constantin--Wu~\citep{CCW2014}, the Boussinesq-type one-dimensional model of Choi--Kiselev--Yao~\citep{CKY2015}, the axisymmetric Euler model of Choi--Hou--Kiselev--Luo--\v{S}ver\'ak--Yao~\citep{CHKLSY2015}, and the Hou--Liu one-dimensional axisymmetric scenario~\citep{HouLuo2014} all illustrate how a reduced stretching law can isolate a finite-time blow-up mechanism.  Recent rigorous PDE singularity and perturbative-stability frameworks include Elgindi--Jeong~\citep{EJ2019,EJ2020}, Chen--Hou~\citep{CH2021,CH2022}, Drivas--Elgindi~\citep{DE2023}, and Elgindi--Pasqualotto~\citep{EP2023}.  Although the present paper does not attempt a full perturbative stability theory, its viewpoint is close in spirit to the broader self-similar and ODE-blow-up stability literature, including Giga--Kohn~\citep{GigaKohn1985}, Giga~\citep{Giga1987}, Merle--Rapha\"el~\citep{MerleRaphael2005}, Rapha\"el--Rodnianski~\citep{RaphaelRodnianski2012}, Collot--Merle--Rapha\"el~\citep{Collot2017}, and Khenissy--Zaag~\citep{Zaag2011}, as well as the local regularity and weighted-estimate tradition represented by Caffarelli--Kohn--Nirenberg~\citep{CKN1982} and Lin~\citep{Lin1998}.

Compared with these works, the present paper uses the squared radial coordinate $(z,q)=(z,r^2)$ to remove the apparent coordinate singularity in the axisymmetric formulation.  The Taylor expansion of $(Bm)$ at $q=1$ then exposes both the exact boundary transport equations and the elliptic obstruction to a closed boundary trace.  The first-order normal-jet closure removes this obstruction and yields a tractable closed system for proving Hou--Luo type finite-time blow-up on the boundary $q=1$.

\medskip
\noindent\textbf{Main results.}
\begin{itemize}
	\item[(i)] We derive the closed parity-reduced subsystems $(B1)$ and $(B2)$ and write them as one unified system $(Bm)$.
	\item[(ii)] We compute the boundary Taylor jet of $(Bm)$ at $q=r^2=1$ and identify the precise normal-jet term that prevents the unrestricted boundary trace from being closed.
	\item[(iii)] Under the first-order normal-jet closure, we obtain the closed unified $(1+1)$D system $(Q0)$ and prove Hou--Luo type finite-time blow-up for smooth periodic data for this closed model.
\end{itemize}

\medskip
\noindent\textbf{Organization.}
Section~\ref{sec:derivation-Bm} derives the unified vorticity--stream formulation $(Bm)$.  Section~\ref{sec:derivation-Q0} derives the boundary jet at $q=r^2=1$ and the closed system $(Q0)$ obtained under the first-order normal-jet closure.  Section~\ref{sec:comparison} compares $(Q0)$ with related one-dimensional models.  Section~\ref{sec:blow-up-proof} proves finite-time blow-up for the periodic version of $(Q0)$.  Section~\ref{sec:conclusion} summarizes the result and the remaining closure issue for the full boundary trace.

\section[The unified system Bm]{The unified vorticity--stream system \texorpdfstring{$(Bm)$}{(Bm)}}\label{sec:derivation-Bm}

We start from the standard vorticity--stream analogy between the 3D axisymmetric Euler equations with swirl, written in cylindrical coordinates $(z,r,\phi)$, and the 2D inviscid Boussinesq equations in the upper half-plane $(z,r)\in\R\times\R^+$.  This analogy is frequently used in the study of singularity formation; see, for example,~\citep{EJ2020,K2018}.

\begin{minipage}{0.42\textwidth} 
	\begin{equation}\label{2DBoussinesq}
		\left\{\begin{aligned}
			&\text{2D inviscid Boussinesq}\\
			& \tfrac{\mathrm{D}}{\mathrm{D}t}\vartheta=0,\quad (z,r)\in\R\times \R^{+}\\
			&\tfrac{\mathrm{D}}{\mathrm{D}t}\tilde\omega=\partial_z\vartheta,\\
			%
%\omega=\partial_x u^y-\partial_y u^x\\
&\tilde\omega
=-\left(\partial_z^2+\partial_r^2\right)\psi,\\
			&\tfrac{\mathrm{D}}{\mathrm{D}t}:=\partial_t +  u^z\partial_z
+u^r\partial_r,\\
			&u^r=-\partial_z\psi,\quad u^z=\partial_r\psi.\\
\end{aligned}\right.
	\end{equation}
\end{minipage}
\begin{minipage}{0.50\textwidth} 
	\begin{equation}\label{3DEuler}
		\left\{\begin{aligned}
			&\text{3D axisymmetric Euler}\\
&\tfrac{\mathrm{D}}{\mathrm{D}t}(rv^\phi)^2=0,\quad\quad (z,r)\in\R\times \R^{+}\\
			& \tfrac{\mathrm{D}}{\mathrm{D}t}\left({\tfrac1r\omega^\phi}\right)=\tfrac{1}{r^4}\partial_z (rv^\phi)^2\\
%
%\omega^\phi=\partial_x v^r-\partial_r v^x\\
&\omega^\phi
=-\left(\partial_z^2 +\partial_r^2+\tfrac{1}{r}\partial_r-\tfrac{1}{r^2}\right)\psi^\phi\\
			&\tfrac{\mathrm{D}}{\mathrm{D}t}:=\partial_t+  v^z\partial_z +  v^r\partial_r,\\
			&v^r=-\tfrac1r\partial_z (r\psi^\phi),	\, v^z=\tfrac1r\partial_r (r\psi^\phi).\\	
		\end{aligned}\right.
	\end{equation}
\end{minipage}\\

It is well known (see, e.g.,~\citep{MB2002}) that the 2D Boussinesq equations~\eqref{2DBoussinesq} share structural features with the 3D axisymmetric Euler equations~\eqref{3DEuler}, at least away from the symmetry axis $r=0$.  Comparing the two systems, $\vartheta$ plays the role of the square of the swirl quantity $rv^\phi$, while the Boussinesq vorticity corresponds to $r^{-1}\omega^\phi$.  The change of variable $q=r^2$ makes this analogy exact at the level of the parity-reduced systems below, without introducing singular factors of $1/r$.

For the 2D inviscid Boussinesq equations we impose the parity ansatz encoded in~\eqref{Boussinesq2}; for the 3D axisymmetric Euler equations we impose the corresponding parity ansatz encoded in~\eqref{Euler2}.\\

	\begin{minipage}{0.48\textwidth}
	\begin{equation}\label{Boussinesq2}
		\left\{
		\begin{aligned}
			\text{2D inviscid}&\text{ Boussinesq}\\
			-\tfrac1ru^r(r)&=:v(r^2),\\
			\tfrac1{r}\vartheta(r)&=:u^2(r^2),\\ 		
			\tfrac1r\tilde\omega(r)&=:\omega(r^2),\\
			\tfrac1r\psi(r)&=:\varphi(r^2),\\
			u^z(r)&=:g(r^2).	
		\end{aligned}
		\right.
	\end{equation}
\end{minipage}
\begin{minipage}{0.48\textwidth}
	\begin{equation}\label{Euler2}
		\left\{\begin{aligned}
			\text{3D }& \text{axisymmetric Euler}\\
			&v(r^2):=-\tfrac1rv^r(r),\\
			&u(r^2):=\tfrac1rv^\phi(r),\\
			&\omega(r^2):=\tfrac1r\omega^\phi(r), \\
			&\varphi(r^2):=\tfrac1r\psi^{\phi}(r),\\
			&g(r^2):=v^z(r).\\
		\end{aligned}\right.
	\end{equation}
\end{minipage}\\

For $m=1$, set $\Theta:=\vartheta/r$.  The notation $\Theta=u^2$ is used only when a smooth square-root branch has been fixed; otherwise one may keep $\Theta$ as the scalar variable.  In the $\Theta$-formulation,
\[
        \frac{\mathrm D}{\mathrm Dt}\Theta=v\Theta,
        \qquad
        \frac{\mathrm D}{\mathrm Dt}\omega=\Theta_z+v\omega .
\]
When $\Theta=u^2$, the first equation is equivalent to $\frac{\mathrm D}{\mathrm Dt}u=\frac12uv$.  This is the only role of the square-root notation in the Boussinesq case.

We can now state the unified formulation.
\begin{theorem}\label{Babc}
Assume the parity reductions encoded in~\eqref{Boussinesq2} for the 2D Boussinesq equations and in~\eqref{Euler2} for the 3D axisymmetric Euler equations with swirl.  In the Boussinesq case, assume also that $\Theta:=\vartheta/r$ is represented on a smooth square-root branch $\Theta=u^2$; equivalently, one may use the scalar $\Theta$-formulation described above.  Then both reductions are transformed, in the square-root notation, into the unified system

\begin{equation}\label{3DEuler2}
		(Bm):\left\{\begin{aligned}
			&\tfrac{\mathrm{D}}{\mathrm{D}t}u\,  
			=\tfrac{m^2}{2}uv,\quad\quad(z,q)\in\R\times\R^+\\
			&\tfrac{\mathrm{D}}{\mathrm{D}t}\omega
			=(u^2)_z+(2-m)v\omega\\
			&\omega=-\left(\partial_{z}^2+4q \partial_{q}^2  + (4+2m) \partial_q\right)\varphi\\
			%
			%\Omega=\alpha_x -1u_s\\
			&v=\varphi_z,\quad g=m\varphi+2q\varphi_q,\\
			&\tfrac{\mathrm{D}}{\mathrm{D}t}:=\partial_t+g\partial_z-2vq\partial_q.
		\end{aligned}\right.
	\end{equation}

Here $(B1)$ is the reformulated 2D Boussinesq system, and $(B2)$ is the reformulated 3D axisymmetric Euler system.

\end{theorem}

\begin{proof}[Proof of Theorem~\ref{Babc}]
Set $q=r^2$.  In both cases the radial velocity has the form
\[
        u^r=-rv \quad \text{or} \quad v^r=-rv,
        \qquad \frac{\mathrm Dq}{\mathrm Dt}=-2qv,
\]
and the axial velocity is
\[
        g=m\varphi+2q\varphi_q,
        \qquad v=\varphi_z.
\]
Hence the material derivative becomes
\[
        \frac{\mathrm D}{\mathrm Dt}
        =\partial_t+g\partial_z-2vq\partial_q .
\]
For the Boussinesq case, $m=1$, the definitions give
$\vartheta=ru^2$, $\tilde\omega=r\omega$, and $\psi=r\varphi$.  Therefore
\[
        \frac{\mathrm D}{\mathrm Dt}(ru^2)=0
        \quad\Longrightarrow\quad
        \frac{\mathrm D}{\mathrm Dt}(u^2)=vu^2,
\]
which is equivalently $\frac{\mathrm D}{\mathrm Dt}u=\frac12uv$ when the chosen square-root branch is used.  Moreover,
\[
        \frac{\mathrm D}{\mathrm Dt}(r\omega)=\partial_z(ru^2)
        \quad\Longrightarrow\quad
        \frac{\mathrm D}{\mathrm Dt}\omega=(u^2)_z+v\omega .
\]
Finally,
\[
        -\frac1r(\partial_z^2+\partial_r^2)(r\varphi)
        =-\bigl(\partial_z^2+4q\partial_q^2+6\partial_q\bigr)\varphi .
\]
This is exactly \eqref{3DEuler2} with $m=1$.

For the axisymmetric Euler case, $m=2$, the definitions give
$rv^\phi=q u$, $\omega^\phi=r\omega$, and $\psi^\phi=r\varphi$.  Using the signed transported angular-momentum variable $rv^\phi=qu$ gives
\[
        \frac{\mathrm D}{\mathrm Dt}(qu)=0
        \quad\Longrightarrow\quad
        \frac{\mathrm D}{\mathrm Dt}u=2uv.
\]
The vorticity equation becomes
\[
        \frac{\mathrm D}{\mathrm Dt}\omega
        =\frac1{r^4}\partial_z(q^2u^2)=(u^2)_z,
\]
and the elliptic identity is
\[
        -\frac1r\left(\partial_z^2+\partial_r^2+\frac1r\partial_r-\frac1{r^2}\right)(r\varphi)
        =-\bigl(\partial_z^2+4q\partial_q^2+8\partial_q\bigr)\varphi .
\]
This is \eqref{3DEuler2} with $m=2$.  Combining the two cases gives the stated unified system.
\end{proof}

\begin{remark}
The system~\eqref{3DEuler2} contains no coordinate singularity factor $1/r=1/\sqrt q$.

For the 3D axisymmetric Euler case, the usual pole conditions at $r=0$ are automatic once the transformed variables are smooth and bounded up to $q=0$.  Indeed, for $q=r^2$ one has
\begin{equation}\label{BC1}
\begin{aligned}
        v^r(z,r)&=-r v(z,q),\\
        v^\phi(z,r)&=r u(z,q),\\
        \partial_r v^z(z,r)&=2r\partial_q g(z,q).
\end{aligned}
\end{equation}
Hence $v^r$, $v^\phi$, and $\partial_r v^z$ vanish at the symmetry axis.  Moreover,
\begin{equation}\label{BC2B}
\begin{aligned}
\omega^\phi(z,r)
&=\partial_z v^r(z,r)-\partial_r v^z(z,r)\\
&=r\left(-\partial_z v(z,q)-2\partial_q g(z,q)\right)\\
&=r\omega(z,q),
\end{aligned}
\end{equation}
where the last identity is exactly the axisymmetric elliptic relation in the variables $(z,q)$.  Thus $\omega^\phi(z,0)=0$ as well.  No additional pole condition is imposed on $\omega$, $u$, or $\partial_z\varphi$ beyond regularity in the transformed variables.
\end{remark}

\section[The boundary-jet system Q0]{The boundary-jet system \texorpdfstring{$(Q0)$}{(Q0)} at \texorpdfstring{$q=1$}{q=1}}\label{sec:derivation-Q0}

We now place \eqref{3DEuler2} on the strip $(x,q)\in\mathbb R\times[0,1]$ and impose the no-flow condition on the outer boundary $q=1$:
\begin{equation}\label{3DEuler2DBoussinesqVS}
\left\{
\begin{aligned}
&\frac{\mathrm D}{\mathrm Dt}u=\frac{m^2}{2}vu,\qquad (x,q)\in\mathbb R\times[0,1],\\
&\frac{\mathrm D}{\mathrm Dt}\omega=(u^2)_x+(2-m)v\omega,\\
&\omega=-\left(\partial_x^2+4q\partial_q^2+(4+2m)\partial_q\right)\varphi,\\
&v=\varphi_x,\qquad g=m\varphi+2q\varphi_q,\\
&\frac{\mathrm D}{\mathrm Dt}=\partial_t+g\partial_x-2vq\partial_q,\\
&v(x,1,t)=0.
\end{aligned}\right.
\end{equation}

Near $q=1$ write the boundary Taylor expansion
\begin{equation}\label{fghqExpansion2b}
\left\{
\begin{aligned}
\omega(x,q,t)&=\omega(x,t)+(q-1)\omega^{(1)}(x,t)+O((q-1)^2),\\
u(x,q,t)&=\rho(x,t)+(q-1)\rho^{(1)}(x,t)+O((q-1)^2),\\
\varphi(x,q,t)&=\phi^{(0)}(x,t)+(q-1)\phi(x,t)+\frac12(q-1)^2\phi^{(2)}(x,t)+O((q-1)^3).
\end{aligned}\right.
\end{equation}
The no-flow condition gives
\begin{equation}\label{uDef}
        v(x,1,t)=\phi^{(0)}_x(x,t)=0.
\end{equation}
Thus $\phi^{(0)}$ is independent of $x$; using the standard boundary stream-function normalization, we set $\phi^{(0)}=0$.  Consequently,
\begin{equation}\label{boundary-g}
        g(x,1,t)=2\phi(x,t).
\end{equation}
Restricting the two transport equations in \eqref{3DEuler2DBoussinesqVS} to $q=1$ gives
\begin{equation}\label{boundary-transport}
\left\{
\begin{aligned}
&\omega_t+2\phi\omega_x=(\rho^2)_x,\\
&\rho_t+2\phi\rho_x=0.
\end{aligned}\right.
\end{equation}
The elliptic equation gives the boundary-jet relation
\begin{equation}\label{boundary-elliptic-jet}
        \omega+4\phi^{(2)}+2(m+2)\phi=0.
\end{equation}
Thus the unrestricted boundary trace is not closed unless the next normal jet $\phi^{(2)}=\varphi_{qq}(x,1,t)$ is also specified.  The closed Hou--Luo type boundary system studied below is obtained by the first-order normal-jet closure
\begin{equation}\label{boundary-linear-closure}
        \phi^{(2)}(x,t)=0,
\end{equation}
that is, by truncating the normalized stream-function Taylor expansion after its first nontrivial normal derivative $\varphi_q(x,1,t)$.  Under this closure, \eqref{boundary-transport}--\eqref{boundary-elliptic-jet} reduce to
\begin{equation}\label{EB2c}
\left\{
\begin{aligned}
&\omega_t+2\phi\omega_x=2\rho\rho_x,\\
&\rho_t+2\phi\rho_x=0,\\
&\omega+2(m+2)\phi=0.
\end{aligned}\right.
\end{equation}
Setting $U=2\phi$ and $\theta=\rho^2$, and then renaming $U$ as $u$, gives the unified closed $(1+1)$D system
\begin{equation}\label{eq:Q0}
\left\{
\begin{aligned}
&\omega_t+u\omega_x=\theta_x,\\
&\theta_t+u\theta_x=0,\\
&u=-\frac1{m+2}\omega.
\end{aligned}\right.
\end{equation}

\begin{remark}[A linear normal-jet closure]\label{rem:linear-normal-jet-closure}
The first-order normal-jet closure \eqref{boundary-linear-closure} is the case $a=0$ of the more general linear closure
\begin{equation}\label{boundary-a-closure}
        \phi^{(2)}(x,t)=a\phi(x,t),\qquad a\in\mathbb R,
\end{equation}
that is, $\varphi_{qq}(x,1,t)=a\varphi_q(x,1,t)$.
Then \eqref{boundary-elliptic-jet} becomes
\[
        \omega+2(2a+m+2)\phi=0 .
\]
Thus, after setting $U=2\phi$ and $\theta=\rho^2$, the same two transport equations close with the local law
\[
        U=-\frac1{2a+m+2}\omega .
\]
If $2a+m+2>0$, this is the same one-dimensional system as \eqref{eq:Q0}, with the coefficient $(m+2)^{-1}$ replaced by $(2a+m+2)^{-1}$.  The blow-up proof in Section~\ref{sec:blow-up-proof} applies verbatim to this one-parameter family.  We keep $a=0$ as the main case because it is the simplest Taylor truncation of the boundary jet.
\end{remark}

\section{Comparison}\label{sec:comparison}

Choi--Hou--Kiselev--Luo--\v{S}ver\'ak--Yao~\citep{CHKLSY2015} gave a detailed summary of several one-dimensional models in the literature and explained how these models were proposed and evolved.  Table~\ref{visina9} follows that comparison: rows 1--6 summarize earlier models, including the Hou--Luo boundary scenario~\citep{LH2014a,LH2014b}, while row 7 adds the closed boundary-jet system \eqref{eq:Q0}.
 
\begin{table}[H]
	\centering
	\small
	\begin{tabularx}{\linewidth}{|P{0.025\linewidth}|P{0.30\linewidth}|P{0.145\linewidth}|P{0.20\linewidth}|P{0.205\linewidth}|} 
		\hline
		&Model&{\small Biot-Savart law}&{\small Dynamical equations}& {\small Regularity of the solutions}\\
		\hline		
1&{\small Constantin et al. (1985) ~\citep{CLM1985},
	{\small analogy of vortex stretching without transport term}}
&$u_x=H\omega$  &$\omega_{t}=u_x\omega $ &{\small finite-time blow-up from smooth data }\\
\hline
2&{\small De Gregorio (1990) ~\citep{DeGregorio1990},
3D Euler analogy}
&$u_x=H\omega$  &$\omega_{t}+u\omega_{x}=u_x\omega $  &{\small global existence and regularity unknown}\\
\hline
3&{\small Cordoba et al. (2005) ~\citep{CCF2005}, SQG analogy}
&$u=H\omega$  &$\omega_{t}+u\omega_{x}=0$ & finite-time blow-up from smooth data \\
\hline
4&{\small Okamoto et al. (2008) ~\citep{OSW2008},
	a generalized De Gregorio model}
&$u_x=H\omega$  &$\omega_{t}+au\omega_{x}=u_x\omega$ & finite-time blow-up from smooth data \\
\hline
	5&{\small HL-model, Hou and Luo (2013) ~\citep{HL2013},
	2D Boussinesq/3D axisymmetric
	Euler analogy}&$u_x=H\omega$&$\omega_{t}+u\omega_{x}=\theta_x$,\quad$\theta_{t}+u\theta_{x}=0$& {\small finite-time blow-up from
			smooth data}\\
		\hline		
		6&{\small CKY-model, Choi et al. (2015) ~\citep{CKY2015},
		simplified HL-model}
	&$-u(x)=x\int_x^\infty\tfrac{\omega(y)}{y}\mathrm{d}y$&$\omega_{t}+u\omega_{x}=\theta_x$,\quad$\theta_{t}+u\theta_{x}=0$& {\small finite-time blow-up from
			smooth data}\\
		\hline
		7&{\small {Our unified boundary-jet 1D system }{\small for Hou--Luo type blow-up on the boundary}} &$u=-\tfrac{1}{m+2}\omega$; (Q0):~\eqref{eq:Q0}&
		$\omega_t+u\omega_{x}=\theta_x$,\quad$\theta_t+u\theta_{x}=0$& {\small finite-time blow-up from
			smooth data}\\
\hline
	\end{tabularx}
	\caption{Comparison of the unified boundary-jet system in row 7 with one-dimensional models summarized by Choi et al.~\citep{CHKLSY2015}.}\label{visina9}
\end{table}
\begin{remark}
The main structural difference between the unified boundary-jet system in row 7 and the models in rows 1--6 is the form of the Biot--Savart law.  
	
More precisely, in the boundary-jet derivation of \eqref{eq:Q0}, the local relation $u=-\frac1{m+2}\omega$ comes from the elliptic boundary-jet relation \eqref{boundary-elliptic-jet} together with the closure \eqref{boundary-linear-closure}.  Under the generalized closure in Remark~\ref{rem:linear-normal-jet-closure}, the same derivation replaces $m+2$ by $2a+m+2$.  By contrast, the earlier models use nonlocal Biot--Savart laws. 

\end{remark}

\section[Finite-time blow-up]{Finite-time (Hou--Luo type) blow-up on the boundary in the periodic case}\label{sec:blow-up-proof}

In this section we prove finite-time blow-up for the closed system \eqref{eq:Q0} on a periodic interval of length $L$:
\begin{equation}\label{likeCKY}
\left\{
\begin{aligned}
&\omega_t+u\omega_x=\theta_x,
        \qquad x\in\mathbb R/L\mathbb Z,
        \quad t\in[0,T),\\
&\theta_t+u\theta_x=0,\\
&u=-c\omega,
        \qquad c>0.
\end{aligned}\right.
\end{equation}
For \eqref{eq:Q0}, $c=(m+2)^{-1}$.  For the generalized linear closure \eqref{boundary-a-closure}, $c=(2a+m+2)^{-1}$, and the condition $2a+m+2>0$ is exactly the positivity condition $c>0$.

The system has the same dimensional scaling as the inviscid Euler equations,
\begin{equation}\label{likeCKY2}
\omega(x,t)\mapsto \frac\mu\lambda\omega(\lambda x,\mu t),\qquad
u(x,t)\mapsto \frac\mu\lambda u(\lambda x,\mu t),\qquad
\theta(x,t)\mapsto \frac{\mu^2}{\lambda^2}\theta(\lambda x,\mu t),
\end{equation}
and it is invariant under adding a constant to $\theta$.

Using $u=-c\omega$, the first equation in \eqref{likeCKY} is equivalent to
\[
        u_t+uu_x=-c\theta_x .
\]
Multiplying by $u$ and using $\theta_t+u\theta_x=0$ gives the local conservation law
\begin{equation}\label{likeCKY3}
        \frac12(u^2)_t-c\theta_t=-\frac13(u^3)_x .
\end{equation}
Consequently, for periodic data,
\begin{equation}\label{likeCKY4}
        E(t):=\int_{-L/2}^{L/2}\left(\frac12u^2-c\theta\right)\,dx
\end{equation}
is conserved as long as the solution remains smooth.

Define
\begin{equation}\label{Ft}
        F(t):=c\int_0^{L/2}\frac{\omega(x,t)}{x}\,dx,
        \qquad
        G(t):=c\int_0^{L/2}\frac{\theta_x(x,t)}{x}\,dx .
\end{equation}
The proof below uses $F(t)$.

\begin{theorem}\label{G5blowup}
Let $\omega_0,\theta_0\in C^\infty(\mathbb R/L\mathbb Z)$.  Assume that $\omega_0$ and $\theta_{0x}$ are odd with respect to both $x=0$ and $x=L/2$, and that
\[
        \omega_0(x)\ge0,
        \qquad \theta_{0x}(x)\ge0,
        \qquad 0\le x\le L/2,
\]
and $F(0)>0$.  Then any smooth solution of \eqref{likeCKY} with these data blows up in finite time.  More precisely, its maximal smooth lifespan satisfies
\[
        T_{\max}\le \frac{L}{F(0)} .
\]
In the normalization $L=2$, this gives $T_{\max}\le 2/F(0)$.
\end{theorem}

\begin{remark}
Choi--Hou--Kiselev--Luo--\v{S}ver\'ak--Yao~\citep{CHKLSY2015} observed that, for the periodic Hou--Luo model, the main quantity in their proof is essentially the functional $G(t)$ in \eqref{Ft}.  They also noted that, in view of the method of C\'ordoba--C\'ordoba--Fontelos~\citep{CCF2005}, a proof based on $F(t)$ would look natural.  For the local Biot--Savart law in \eqref{likeCKY}, the $F(t)$ argument closes directly.
\end{remark}

\begin{proof}[Proof of Theorem~\ref{G5blowup}]
The assumptions on $\theta_{0x}$ imply that $\theta_0$ is even with respect to both endpoints.  These endpoint symmetries are preserved by \eqref{likeCKY}: if $\omega$ is odd and $\theta$ is even with respect to an endpoint, then $u=-c\omega$ is odd, $u\omega_x$ and $\theta_x$ are odd, and $u\theta_x$ is even.  In particular, $u(0,t)=u(L/2,t)=0$, so the interval $[0,L/2]$ is invariant under the characteristic flow.  Moreover,
\[
        (\theta_x)_t+u(\theta_x)_x=-u_x\theta_x,
        \qquad
        \omega_t+u\omega_x=\theta_x .
\]
Thus $\theta_x$ preserves its sign along characteristics, and then $\omega$ remains nonnegative on $[0,L/2]$ as long as the smooth solution exists.  Since $\omega$ is odd at both endpoints, $\omega(0,t)=\omega(L/2,t)=0$.

Differentiating $F(t)$ and using $u=-c\omega$ gives
\begin{align*}
\dot F(t)
&=c\int_0^{L/2}\frac{\omega_t}{x}\,dx                                      \\
&=c\int_0^{L/2}\frac{\theta_x}{x}\,dx
  +c^2\int_0^{L/2}\frac{\omega\omega_x}{x}\,dx                              \\
&=c\int_0^{L/2}\frac{\theta_x}{x}\,dx
  +\frac{c^2}{2}\left[\frac{\omega^2}{x}\right]_0^{L/2}
  +\frac{c^2}{2}\int_0^{L/2}\frac{\omega^2}{x^2}\,dx .
\end{align*}
The boundary term vanishes because $\omega$ is smooth and odd at $0$ and $L/2$, while the first integral is nonnegative.  Hence
\begin{equation}\label{dFdt}
        \dot F(t)\ge \frac{c^2}{2}\int_0^{L/2}\frac{\omega^2}{x^2}\,dx .
\end{equation}
By Cauchy--Schwarz,
\begin{equation}\label{F-cauchy}
        F(t)^2
        \le c^2\frac{L}{2}\int_0^{L/2}\frac{\omega^2}{x^2}\,dx .
\end{equation}
Combining \eqref{dFdt} and \eqref{F-cauchy}, we obtain
\begin{equation}\label{F-riccati}
        \dot F(t)\ge \frac1L F(t)^2 .
\end{equation}
Therefore
\[
        \frac1{F(t)}\le \frac1{F(0)}-\frac{t}{L}.
\]
The right-hand side reaches zero at $t=L/F(0)$.  Since $F(t)$ is finite for every smooth odd solution, the smooth solution cannot persist beyond this time.
\end{proof}

\section{Conclusion}\label{sec:conclusion}%=====================================
We derived the unified parity-reduced system $(Bm)$ from the 2D inviscid Boussinesq equations and the 3D axisymmetric Euler equations with swirl.  The squared radial variable $q=r^2$ removes the apparent axis singularity from the axisymmetric formulation and encodes the Boussinesq and Euler cases through the parameter $m=1,2$.

At the boundary $q=1$, the Taylor expansion of $(Bm)$ gives exact boundary transport equations but also reveals the normal-jet obstruction $\varphi_{qq}(x,1,t)$ in the elliptic relation.  Under the first-order normal-jet closure $\varphi_{qq}(x,1,t)=0$, the boundary jet becomes the closed unified system $(Q0)$ with local velocity law $u=-(m+2)^{-1}\omega$.  More generally, the linear closure $\varphi_{qq}(x,1,t)=a\varphi_q(x,1,t)$ gives the same closed transport system with $u=-(2a+m+2)^{-1}\omega$ whenever $2a+m+2>0$.  For these closed $(1+1)$D systems, we proved finite-time blow-up from smooth periodic data satisfying the symmetry and sign assumptions of Theorem~\ref{G5blowup}.

The result should therefore be interpreted as a rigorous blow-up theorem for a boundary-jet closure derived from $(Bm)$, not as a blow-up theorem for the full Boussinesq or Euler systems.  A natural next step is to understand whether the normal jet $\varphi_{qq}(x,1,t)$ can be controlled, closed, or coupled to higher boundary jets in a way that preserves the same boundary blow-up mechanism.

\section{Acknowledgements}\label{sec:acknowledgements}
The author thanks Prof. Zixiang Zhou of the Department of Mathematics at Fudan University and Prof. Jie Qin of the Department of Mathematics at the University of California, Santa Cruz for their continuous support and encouragement over the years.  The author also thanks colleagues and the broader PDE and fluid-dynamics community for stimulating discussions on axisymmetric Euler equations and CLM-type models.  The author acknowledges helpful drafting, editing, and consistency-checking assistance from OpenAI's GPT-5.5 Pro through ChatGPT during the preparation of this revision; responsibility for the mathematical correctness of the manuscript remains with the author.

\newpage
%
%---------------------------------------

%---------------------------------------
\end{document}